\documentclass[11pt]{amsart}
\usepackage{latexsym}
\usepackage{amsfonts, amsmath, amscd}
\usepackage{amssymb}
\usepackage[all]{xypic}
\usepackage{color}
\usepackage{enumerate}
\usepackage{verbatim}
\usepackage{amsthm}    
\usepackage{hyperref}
\hypersetup{colorlinks=false}
\usepackage{wrapfig}
\usepackage{graphicx, psfrag}
\usepackage{pstool}
\usepackage{pinlabel}   

     \usepackage{bm}
    
   \usepackage[normalem]{ulem}

\def\cal{\mathcal}

\def\de{\delta}

\def\text#1{{\em #1}}

\def\la{\lambda}

\def\ep{\varepsilon}

\def\be{\begin{equation}}
\def\ee{\end{equation}}
\def\bear{\begin{eqnarray}}
\def\eear{\end{eqnarray}}
\def\best{\begin{eqnarray*}}
\def\eest{\end{eqnarray*}}

\newtheorem{theorem}{Theorem}[section]
\newtheorem{prop}[theorem]{Proposition}

\newtheorem{lemma}[theorem]{Lemma}
\newtheorem{cor}[theorem]{Corollary}

\newtheorem{defn}[theorem]{Definition}

\newtheorem{remark}[theorem]{Remark}

\newtheorem{example}[theorem]{Example}

\def\ra{\rightarrow}

\def\r#1{\right#1}
\def\l#1{\left#1}

\def\phi{\varphi}
\def\ti{\times}

\def\del{\overline \partial}

\def\F{{\Bbb F}}

\def\M{{\cal M}}

\def\ev{\mathrm{ev}}

\def\wt#1{\widetilde{#1}}
\def\wh#1{\widehat{#1}}
\def\ov#1{\overline{#1}}

\def\w{\omega}

\def\La{\Lambda}

\def\E{\mathcal E}
\def\F{\mathcal F}

\def\D{\mathbf D}

 \usepackage{hyperref}
 \hypersetup{colorlinks=false}

\def\Ax{{\cal A}\hskip-1pt{\it pprox}}

    \def\!{|\hspace{-1.5pt}|\hspace{-1.5pt}|}
    
\def\cit#1{\cite[(#1)]{ip}}

\title{\bf Corrigendum: The symplectic sum formula for Gromov-Witten invariants\vskip.2in}
\author{ Eleny-Nicoleta Ionel}
\address{Stanford University,
 Palo Alto, CA, USA.}
\email{ionel@math.stanford.edu}

\author{Thomas H. Parker}
\address{Michigan State University, East Lansing MI, USA. }
\email{parker@math.msu.edu}

\addtocounter{section}{0}

\begin{document}

\begin{abstract}
We correct an error and an oversight  in  \cite{ip}.  The sign of the curvature in (8.7)  is wrong, requiring a new proof of Proposition~8.1.     Several lemmas address only the basic case of maps  with intersection multiplicity $s={\mathbf 1}$; the general case follows by applying  the pointwise   estimates in \cite{ip} with a modified Sobolev norm.
\end{abstract}
   
\maketitle



\vspace{5mm}

\section{{\bf \large Sobolev Norms }}
\setcounter{theorem}{0}

In \cite{ip}, portions of Sections~6-8 are valid only for  maps  with intersection multiplicity $s={\mathbf 1}$.  To cover  maps with multiplicity vector $s=(s_1, \dots, s_\ell)$, we 
modify the  Sobolev norms in \cit{6.9} by setting 
\begin{align}
\label{new6.9}
\|\zeta\|^p_{m,p,s} & \ =\   \int_{C_\mu}    \Big(  |\nabla^m\zeta|^p + |\zeta|^p\Big)\   \rho^{-\de p/2} \nonumber \\
& \hspace{2cm} + \sum_k \int_{C_\mu\cap B_k(1) } \   \Big( |\rho^{1-s_k}\nabla^m(\zeta^N)|^p + |\rho^{1-s_k}\zeta^N|^p \Big)\   \rho^{-\de p/2}. 
\end{align}
With this revision, the norms $\!(\xi, h)\!_m$ and $\!\eta\!_m$ are again  defined by formulas \cit{6.10} and \cit{6.11}.  For $s={\bf 1}$, the above norm is uniformly equivalent to  the norm of \cit{6.9}.  For general $s$, it now has a stronger  weighting  factor of $\rho^{1-s_k}$ on the normal components near  each node with multiplicity $s_k\ge 2$.  Accordingly, one must verify that  Lemmas~6.9 and 7.1, Proposition~7.3 and Lemma~9.2 continue to hold for this new norm.  This is easily done using the pointwise estimates  already appearing  in the proofs, as follows.

\vspace{6mm}

\noindent{\bf Modifications to Section~6.}\ \   Lemma~6.8 is unaffected by the change in norms. The statement of Lemma~6.9 remains valid for the new norms with a slight modification to the exponent in its conclusion:
\best
\label{newlemma6.9}
\!\overline{\partial}_{J_F} F -\nu_F\!_0 \ \leq\  c \,|\la|^{\frac1{5|s|}}
\eest
with $c$ uniform on each compact set ${\mathcal K}_{\delta_0}$. (Throughout \cite{ip}, the $\delta$ indexing the sets ${\mathcal K}_{\delta}$ of \cit{3.11}
is unrelated to the exponent $\delta$ in \eqref{new6.9}.) The proof of Lemma~6.9 is modified as follows.

\begin{proof}[Proof of Lemma~6.9]
Set $\Phi_F= \del_{J_F}-\nu_F$ and follow the proof in \cite{ip} until (6.14).  Outside $\rho \le \rho_0$, the new norm is uniformly equivalent to old one, so \cit{6.13} continues to hold for $\Phi=\Phi_F$. Again we focus on the half $A_+$ of one $A_k$ where $|w_k|\le |z_k|$ where, after omitting subscripts, $F$ is given by \cit{6.14}:
\bear
\label{new6.14}
F=\left((1-\beta)h^v, \; az^s (1+(1-\beta)h^x), \; bw^s(1+ (1-\beta)h^x)^{-1}\right)
\eear
where $\beta, h^v$ and $h^x$ are functions of the coordinate $z$ on $C_1\subset C_0$ and $w=\mu/z$.

Introduce the map $\wt{f}:C_0\to X$  given by \eqref{new6.14} with the last entry replaced by zero.   Noting that $\Phi_f=0$ (because $f$ is $(J, \nu)$-holomorphic), we can write
\bear\label{new.6.sumofPhis}
\Phi_F\ =\ (\Phi_F-\Phi_{\tilde{f}}) + (\Phi_{\tilde{f}}-\Phi_{f}).
\eear
To complete the proof, we will bound the two expressions on the righthand side using the following facts, which hold for small $\la$:
\begin{enumerate}
 \setlength\itemsep{2mm}
\item[(i)]  \cite[Lemma~6.8d]{ip}  implies that $|h^v|+ |dh^v| +|h^x|+|dh^x|  \le c\rho \le \frac12$ on $A_+$. 
\item[(ii)] $|dz|=|z|$ and $|dw|=|w|$ in the cylindrical metric \cit{4.5}, and $0\le \beta \le 1$ and $|d\beta|\le 2$ by \cit{5.11}.
\item[(iii)]  On $A_+$,  $\rho^2=|z|^2+|w|^2\le 2|z|^2$, and hence  $|z|^{-2}\le2\rho^{-2}$ and $\sqrt{|\mu|}\le|z|\le  \rho \leq  \rho_0$.
\item[(iv)] By \cite[Lemma~6.8c]{ip},  $a_k$ and $b_k$ are bounded above and below by positive constants, and hence $|\mu_k^{s_k}|\sim |\la|$ by  \cit{6.3}.\\
\end{enumerate}
Writing \eqref{new6.14} as $(F^v, F^x, F^y)$, Facts~(i)-(iv)  imply the pointwise bounds
\bear
\label{FdFbounds}
|F^v|+|dF^v|\le c\rho, \qquad 
|F^x| + |dF^x|\le c\rho^s, \qquad 
|F^y|+|dF^y|\le  c|w|^s\le C|\la|\rho^{-s}
\eear
for  $s=s_k$,  with constants $c, C$ uniform on ${\mathcal K}_{\delta_0}$.   It follows  that  $|dF|\le 3c\rho$ and, since $J$ and $\nu$ are smooth,  
\bear\label{est.F-f}
|J_F-J_{\wt f}| + |\nu_F-\nu_{\wt f}|\ \le\  c|F-\wt f|\ =\ c|F^y|\  \le\  c|\la|\rho^{-s}.
\eear
(Here, and below, we are updating the constant $c$ as we proceed.)

Now, using the definition of $\Phi_F$, we have
\bear
\label{PhiFFormula1}
2(\Phi_F-\Phi_{\tilde{f}})\  = \ d(F-\tilde{f})   +(J_F-J_{\tilde{f}})dF j  +J_{\tilde{f}}(dF-d\tilde{f})j-2 (\nu_F-\nu_{\wt f}), 
\eear
with $F-\wt f=F^y$.  Estimating each term, one sees that the above bounds  imply that
\bear
\label{err.F-wtf}
|\Phi_F-\Phi_{\wt f}|\le c |\la|\rho^{-s} \quad\mbox{ so} \quad \rho^{1-s}|(\Phi_F-\Phi_{\wt f})^N| \ \le\  c|\la| \rho^{1-2s}
\eear
on $A_+$.  Applying the second integral in \cit{5.10}, noting  that $s \ge 1$ and that $\rho^2\ge|\mu|$ on $A_+$  yields
\bear
\label{6.9firsthalf}
\! \Phi_F-\Phi_{\wt f}\!_{0,A_+}\ \le\  c |\la| |\mu|^{\frac1 2 (1-2s-\de/2)}\ \le\  c|\la|^{\frac1{3s}}
\eear
where the last inequality uses (iv) above and the fact that $0<\delta<\frac16$. By symmetry, a similar estimate holds on the other half of $A_k$.  Hence \eqref{6.9firsthalf} holds on the entire set  $A\subset C_0$ where $\rho\le\rho_0$ with a revised constant $c$ and the exponent replaced by $\frac1{3|s|}$ (since $|s|\ge s_k$ for all $k$).

It remains to  estimate the last term in \eqref{new.6.sumofPhis}.  On $A_+$,  the
difference between $f$ and $\wt f$, namely
\best
f-\wt f =\beta (h^v, az^s (1+ h^x), 0),
\eest
 is supported in the region $\rho\le 2|\mu|^{1/4}$. Again expanding as in \eqref{PhiFFormula1} and using (i)-(iv), one obtains  
\best
|\Phi_{\wt f} -\Phi_{f}|\ \le\ |f-\wt f |+|df-d\wt f|+ |J_f-J_{\wt f}| + |\nu_f-\nu_{\wt f}|\ \le\  c\rho 
\eest
Similarly, using (i)-(iv),  the normal component $(\wt f- f)^N = -\beta az^s (1+ h^x)$ satisfies: 
\best
|(\wt f- f)^N|+ |d(\wt f- f)^N|\ \le\   c \rho^s.  
\eest
Next observe that the images of  $f$ and $\wt f$ both lie in $Z_0=X\cup Y$  and, as in \cit{6.6},  $J$ preserves the  normal subbundle $N_0$ to $V$ in $Z$ along $Z_0$. Also noting that \eqref{FdFbounds} implies  that $|(d\wt f)^N| \le c\rho^s$, one sees that
$$
\left|((J_{f}-J_{\wt f})\circ d \wt f)^N +(J_f ( d\wt f - d f))^N \right|\ \le \ 
\left|(J_{f}-J_{\wt f})\circ (d \wt f)^N\right| +  \left|J_f( d\wt f - d f)^N \right|\ \le\  c \rho^s.
$$
Because $\nu^N$ vanishes along $V$, there is also a bound $|\nu^N_f |\le C|f^N|\le c\rho^s$; the same is true for $\wt f$, and therefore $ |\nu^N_{\wt f}- \nu_f^N|\le c\rho^s$. Expanding $\Phi_{\wt f}-\Phi_f$ as in \eqref{PhiFFormula1} and using above estimates yields
\bear
|\Phi_{\wt f}-\Phi_{f}|+\rho^{1-s}|(\Phi_{\wt f}-\Phi_{f})^N| \ \le\  c\rho
\eear 
with the lefthand side supported in the region $\rho\le 2|\mu|^{1/4}$ in $A_+$ and, by symmetry, in $A_k$ for each $k$.    Applying the first integral  in \cit{5.10} bounds the integrals in the norm \eqref{new6.9} on the union $A$ of the $A_k$.  Again using the facts that  $|\la|\sim |\mu_k|^{s_k}$,  $0<\delta<\frac16$ and $s_k\le |s|$, one obtains the bound
\bear
\label{6.9secondhalf}
\!\Phi_{\wt f}-\Phi_f\!_{0,A}\ \le\  c |\mu|^{\frac14 (1-\de/2)}\ \le\  c|\la|^{\frac1{5|s|}}.
\eear
The lemma now follows from  \cit{6.13} and the bounds \eqref{6.9firsthalf} and \eqref{6.9secondhalf} on the norms of the two terms in  \eqref{new.6.sumofPhis}.
\end{proof}

\vspace{6mm}

\noindent{\bf Modifications to Section~7.}\ \   Delete the paragraph that starts after  \cit{7.4} and ends with \cit{7.6}; we no longer need $D_F^*$.

\smallskip

$\bullet$  The conclusion of  Lemma~7.2 should read 
\bear
\label{newLemma7.2eq}
|(\nabla\zeta^V)^N|\le c\rho^s|\zeta^V|, \hspace{1cm}  |L_F^N\zeta^V |\le c\rho^s |\zeta^V|
\eear
 {\em as the proof shows} (keep all powers of $\rho$ in the last line of the two paragraphs of the proof).\\

$\bullet$  The statement  of Proposition~7.3 remains the same after deleting the statement about $D^*_F$.  The proof is unchanged until two lines before \cit{7.10}, at which point we have established the estimate
\bear\label{displ.d}
 |\D_{F, C}(\zeta ,\ov\xi, h)|\ \le\  c|\nabla \zeta| +c\rho( |\zeta|+|\ov \xi|+ |h|) 
\eear
(this is  also \cit{9.7}).  As a special case, we have
\best
|\D_{F, C}(\zeta^N ,0, 0)^N|\ \le \ c \big(|\nabla \zeta^N| +\rho |\zeta^N|\big).
\eest
On the other hand, the normal component is
\best\label{bd.D.N}
\left(\D_{F, C}(\zeta^V ,\ov\xi, h)\right)^N\ = \  \left(L_F(\zeta^V+ \sum \beta_k \ov \xi_k)\right)^N + (J_FdF h)^N.
\eest
Using \eqref{newLemma7.2eq}, the first term on the right is  bounded by $c\rho^s(|\zeta^V|+|\ov \xi|$).  The second is dominated by $|(J_F(dF^V+dF^N))^N|\,|h|$ with $|dF^N|\le c\rho^{s}$ by \eqref{FdFbounds}.  Furthermore, because  $V$ is $J$-holomorphic,  $(Jv)^N=0$  along $V$ for all  vectors $v$ in the $V$ direction.  Hence  $|(J_FdF^V)^N|\le c |F^N| |dF^V|\le c\rho^s$, again using \eqref{FdFbounds}. Altogether, this gives  the following pointwise  bound:
\bear
\label{est.DN}
\left|\D_{F, C}(\zeta ,\ov\xi, h)^N \right| \ \le \ c \left(|\nabla \zeta^N| \, +\, \rho|\zeta^N|\right) + c\rho^s\left(|\nabla \zeta^V|+ |\zeta^V|+|\ov \xi|+ |h|\right).  
\eear
Multiplying both sides of this inequality by $\rho^{1-s-\de/2}$, raising  to the power $p$ and integrating shows that \cit{7.10}   holds in  the new norms.  
The proof is completed as before. 

\vspace{6mm}

\noindent{\bf Modifications to Section~8.}\ \  See Section~2 below.

\vspace{6mm}

\noindent{\bf Modifications to Section~9.}\ \   Section~9  uses Proposition~8.1, but not its proof: the existence of a right inverse is needed, but nothing about its construction. Switching to the new norms \eqref{new6.9} does not affect  Proposition~9.1, and requires only small modifications to the proofs of  Lemma~9.2 and Proposition~9.3.  \\

$\bullet$ {\bf Lemma 9.2}:   The statement of Lemma~9.2 remains the same.  For the proof, again let  $B_k$ be the region around the $k$th node where $\rho\leq |\mu|^{1/4}$ and let $A_k$ be the larger region where $\rho\le \rho_0$.  The new norms \eqref{new6.9}  differ from the old norms only in the weighting of the normal components near the nodes.  Thus to show that Lemma~9.2 holds in the new norms we need only bound the $L^p$ integral of the weighted normal components
$$
|\rho^{1-s_k-\de/2}\D_F(\xi, h)^N |,
$$
 first on $A_k\setminus B_k$, then on $B_k$.  Fix $k$ and write $\mu=\mu_k$ and $s=s_k$.

 Follow the existing proof until two lines below \cit{9.5}, at which point we have  identified   sections of $F^*TZ_\la$ over  $C_\mu\setminus B$ with sections of $f^*TZ_0$  over $C_0\setminus B$, and  established the estimate:
\bear
\label{9.2second}
|D_F(\xi, h)-D_f(\xi, h)|\ \le\  |(L_F-L_f)(\xi)|+ (|J_F-J_f|+|dF-df|)|h|
\eear
 with $D_f(\xi, h)=0$.  Formula \cit{1.11}  shows that $L_f$ is a first order differential operator of the form
$$
L_f(\xi) \ =\ A_f(\nabla\xi) + B_f(\xi) 
$$
whose coefficients $A_f$ and $B_f$ are continuous functions of $f$ and $df$.  Hence
\bear\label{9.2.third}
|L_F\xi-L_f\xi|\ \le\ c\, \left(|F-f| + |dF-df|\right)\cdot  \left(|\nabla\xi| + |\xi|\right)
\eear
for some constant $c$. But in the region $A_k\setminus B_k$ we have $F-f=F-\wt f = F^y$.  Then
 \eqref{FdFbounds}  shows  that
the $C^1$ distance between $F_t$ and $f_t$ is dominated by $|\la|\rho^{-s}$, so  \eqref{9.2second} implies the bound 
$$
|D_F(\xi, h)| \ \le\  c |\la|\rho^{-s} \left(|\nabla\xi| + |\xi|+|h|\right)
$$
on $A_k\setminus B_k$.   After expanding $\xi$ as in \cit{6.8} and noting that $|\nabla (\beta_k \xi_k)|\le c |\ov\xi|$ by the estimate preceding \cit{7.9}, the above bound simplifies to 
\bear
 \label{9.new9.6}
 |\D_F (\zeta,\ov\xi, h)|\ \le\  \ c|\la|\rho^{-s }\left(|\nabla\zeta| +|\zeta|+|\ov \xi|+ |h|\right),
 \eear
a mild strengthening of the displayed equation above \cit{9.6}. But in the region  $A_k\setminus B_k$,  we have   $|\mu|^{1/4}\le\rho\le\rho_0$ and $|\la| \sim |\mu|^s$ so   \eqref{9.new9.6} implies 
\bear
\label{1.15B}
|\D_F (\zeta,\ov\xi, h)|+\rho^{1-s}|\D_F (\zeta,\ov\xi, h)^N|\ \le c|\la|^{1/4}\rho^{1+s}\left(|\nabla\zeta| +|\zeta|+|\ov \xi|+ |h|\right).
\eear
Taking the norms defined by \cit{6.10} and  \eqref{new6.9}, shows that \cit{9.6} continues to hold in the new norms.

Now focus on one $B_k$. Proceed as in \cite{ip}, using  the new estimate \eqref{est.DN} to strengthen \cit{9.7}.   For $\xi=\xi^V$,    \eqref{est.DN}  gives
\bear
\label{9.1}
|\D_F (\zeta^V, \ov \xi, h)^N| \ \le \  c\rho^s\l(|\nabla\zeta^V|+ |\zeta^V|+|\ov \xi|+ |h|\r)
\eear
on each $B_k$. For $\xi=\xi^N$, we again have $\ov\xi=0$ and $\xi=\zeta$, so \eqref{est.DN}   and  the last displayed equation on page 988 give
\bear
\label{9.2}
|\D_F (\xi^N, 0)^N|\ \le\   c\left(|\nabla \zeta^N|+\rho|\zeta^N|\right) \  \le \   c\rho^s (|\dot a|+|\dot b|).
\eear
 Combining \eqref{9.1} and \eqref{9.2}  with the argument on top of page 989 shows that the conclusion of the first displayed equation on top of page 989 continues to hold in the new norms. The proof is completed as before. \\
 
$\bullet$ {\bf Proposition~9.3}:  Replace the last 4 lines on page 989 of \cite{ip} by the following:  write $F_n-f_n=(\zeta_n, \bar{\xi}_n)$ in the notation of (6.7) and (6.8).  Then $\bar{\xi}_n\to 0$ because $f_n\to f_0$ in $C^0$.  By Lemma~5.4, the norm $\!f_n\!_1$   on $A_k(\rho_0)$ is bounded by $c\rho_0^{1/6}$.  Inserting the bounds \eqref{FdFbounds}  into \eqref{new6.9} and integrating using \cit{5.10} gives the similar inequality $\!F_n\!_1\le c\rho_0^{1/6}$  on $A_k(\rho_0)$.  Therefore $\!\zeta_n\!_1 \le \!F_n\!_1+\!f_n\!_1 +|\bar{\xi}_n| \le 3c\rho_0^{1/6}$ on 
 $A_k(\rho_0)$ for all large $n$.  Combining $\dots$  {\em Continue at the top of page~990, and change  $2|s| \mapsto 5|s|$ on  page 990, line 14.}

\vspace{5mm}

 \section{{\bf \large Revised Section 8}}

An incorrect formula \cit{7.5} for the adjoint and a   sign error in the curvature formula \cit{8.7} invalidate the proof of Proposition~8.1.    The following  replacement for Section~8 
 retains everything up to and including the statement of Proposition~8.1,  and then gives a new proof of Proposition~8.1. 
  Instead of establishing   eigenvalue estimates, this new proof   transfers the partial right inverse $P$ from the nodal curve $C_0$ to its smoothing $C_\mu$. 
  The proof is then easier, the adjoint $D_F^*$ never appears, and   again the required estimates follow from  pointwise bounds  already  in   \cite{ip}.

\bigskip

Retain the beginning of Section~8 of \cite{ip} up to Proposition~8.1. \\

 To simplify notation, note that for $F\in \Ax_s^{\de_0}(Z_\la)$, \cite[Proposition 7.3]{ip} shows that the linearizations $\D_F$ of \cit{7.4} are uniformly  bounded operators
 $$
 \D_F:\E_F\to \F_F
 $$
between the spaces
\best
\E_F=L_{1;s, 0}(F^*TZ_\la)\oplus T_q V^\ell \oplus T_{C_\mu}\ov\M_{g,n}
\quad\mbox{and} \quad 
  \F_F= L_s(\La^{01}( F^*TZ_\la)),
 \eest
while the linearization $\D_f=\D_{f, C_0}$ of \cit{7.3} at each $f\in\M_s^V(X)\ti_{ev}\M_s^V(Y)$ is a map between the corresponding spaces $\E_f$ and $\F_f$.

\medskip

The aim of this section is to prove the following analytic result.

\begin{prop}
\label{prop8.1}
    For each generic $(J,\nu)\in {\cal J}(Z)$, there are positive constants $\la_0$ and 
$E$  such that, for all  non-zero $\la\le \la_0$, the linearization
$\D_F$ at an approximate map $(F,C_\mu)=F_{f,C_0, \mu}\in
\Ax_s^{\de_0}(Z_\la)^*$  has a right inverse
\bear\label{P.F.defn} 
P_F :\E_F\to \F_F  
\eear
that satisfies $\D_FP_F=id$ and
\bear
\label{Pmu.bounded}
  E^{-1}\,\!\eta\!_0 \;\le \;\!P_F \eta\!_{1} \ \le\  E\,\!\eta\!_0.
\eear
\end{prop}

\bigskip

Proposition~8.1 is proved by constructing an approximation to $P_F$ in the following sense.

\begin{defn}
\label{8.defRRI}
An  {\em approximate  right inverse to ${\mathbf D}_F$} is a  linear map 
$$
A_F: \F_F\to \E_F
$$ 
such that, for all $\eta\in\F_F$,
\bear
\label{8.defroughinverse}
\! D_F A_F \eta - \eta\!_0\  \le\  \tfrac12 \! \eta\!_0 \qquad \mbox{and}\qquad   \!A_F\eta\!_1 \ \le\   C \! \eta\!_0. 
\eear
\end{defn}

\medskip

Such an approximate  right inverse defines an actual right inverse by the formula 
\best
P_F= A_F\sum_{k\ge 0}  (I - D_F A_F)^k.
 \eest 
The bounds \eqref{8.defroughinverse} ensure that this series converges and defines a bounded operator $P_F$, which satisfies $D_FP_F=I$.  Because both $P_F$ and $D_F$ are bounded (cf. \cit{Lemma~7.3}), we have $\!P_F \eta\!_{1} \ \le\  c\,\!\eta\!_0$ and 
$\!\eta\!_0 = \!D_FP_F \eta\!_{0}  \le  c\,\!P_F\eta\!_1$, which gives \eqref{Pmu.bounded}.

Thus Proposition~8.1 follows from the existence of an 
approximate right inverse $A_F$ as in Definition~\ref{8.defRRI},  where the constant $C$ in \eqref{8.defroughinverse} is uniform in $\la$ for small $\la$.  
The  remainder of this section is devoted to constructing such an $A_F$.

\bigskip

We start by observing that, under the hypotheses of Proposition~\ref{prop8.1}, we may assume that $f$ is regular (cf. \cit{Lemma~3.4}).  Thus  $\mathbf D_f:\E_f\to \F_f$ is a bounded surjective map, so has a bounded right inverse $P_f: \F_f \ra \E_f$.  We will use a splicing construction to transfer $P_f$ from an operator on $C_0$ to one on the domain $C_\mu$ of $F$, and show that the resulting operator $A_F$ satisfies \eqref{8.defroughinverse}.
The construction is summarized by the following  (noncommutative) diagram:
 \bear
 \label{diagram8.1}
\xymatrixcolsep{4pc}
\xymatrix{ 
\F_F \ar[d]^{\pi_F} \ar@{.>}[r]^{A_F}& \E_F\ar[r]^{\D_F}&\F_F  \ar[d]_{\pi_F} 
\\
\F_f \ar[r]^{P_f}\ar@ /^1pc/[u]^{\gamma_F} & \E_f\ar[u]_{\Gamma_F} \ar[r]^{\D_f}&\F_f
\ar@ /_1pc/[u]_{\gamma_F}
} 
\eear
Each of the maps $\gamma_F$, $\pi_F$ and $\Gamma_F$ will be defined by regarding the two halves of  $C_\mu$ as a graphs over $C_0$, and similarly regarding $Z_\la$ as graphs over $Z_0$.  The desired approximate right inverse is then defined by 
$$
A_F= \Gamma_F\circ P_f\circ \pi_F.  
$$

\medskip

Our notation for splicing is as in Lemma~9.2 of \cite{ip}. For each $\mu\not= 0$, let 
$$
C_1(\mu)\ =\  C_1\cap \left\{|z|\ge |\mu|^{\frac34}\right\} \qquad   C_2(\mu)\ =\  C_2\cap \left\{|w|\ge |\mu|^{\frac34}\right\}.
$$
and let $C^+_\mu$ and $C^-_\mu$ be the corresponding parts of $C_\mu$  (see the figure). We identify $C_1(\mu)$ with   $C_\mu^+$   by  the projection $(z,w)\mapsto z$. With this identification, $z$ is a coordinate on both $C_1(\mu)$ and $C^+_\mu$ and, similarly,   $w$ is a coordinate on both $C_2(\mu)$  and $C^-_\mu$.

 \begin{center}
 \includegraphics[scale=1]{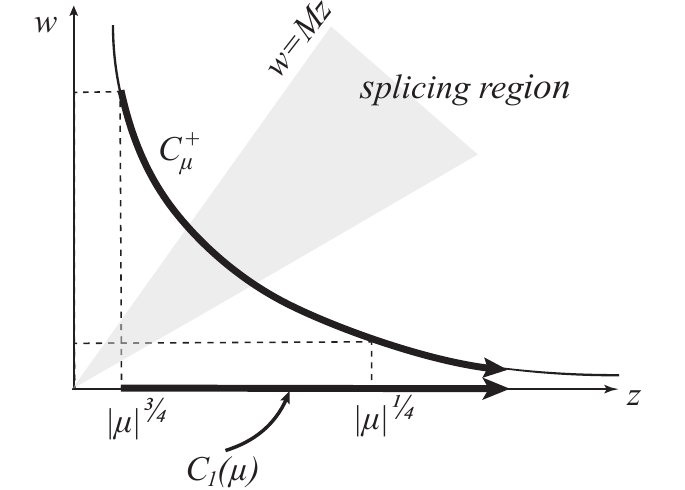}
 \end{center}
 
There is a corresponding picture in the target:  the projections $(v,x,y)\mapsto (v,x)$ and  $(v,x,y)\mapsto (v,y)$ give identifications $Z_\la = X$ in the region $Z_\la^+$ where $|x|\ge |\la|^{3/4}$, and $Z_\la = Y$ in the region $Z_\la^-$  where $|y|\ge |\la|^{3/4}$. This  trivializes $f^* TZ_0$ and $F^* TZ_\la$ inside the coordinate chart $(v,x,y)$.
  
These identifications, together  with \cite[Definition 6.2]{ip}, induce   isomorphisms
\bear\label{D.8.hat.map}
\Omega^{0,q}(C_1(\mu), f^* TX) \ra \Omega^{0,q}(C_\mu^+, F^* TZ_\la) \qquad \mbox{by $\xi_1\mapsto \wh\xi_1$}
\eear
for $q=0, 1$ defined by  $\wh\xi_1(z) =\xi(z)$ in $B_k$  under the above identifications of $C_1(\mu)$ with $C_\mu^+$ and of $X$ with $Z_\la^+$, extended by setting
$\wh\xi_1 =\xi$ outside the union of the balls $B_k$ of radius $2|\mu_k|^{1/4}$ (where  $C_0$ is identified with $C_\mu$ and $F=f$).  Permuting $z\leftrightarrow w$ and $x\leftrightarrow y$ gives  similar isomorphisms
   $$
\Omega^{0,q}(C_2(\mu), f^* TX) \ra \Omega^{0,q}(C_\mu^-, F^* TZ_\la) \qquad \mbox{by $\xi_1\mapsto \wh\xi_2$.}
$$

\bigskip

\begin{lemma}\label{lemma8.3}

For each region $\Omega_M^+$ defined by $M^{-1}|w|\le|z| \le 1$, there are constants $c_M, \la_M>0$ such that
the map \eqref{D.8.hat.map}  satisfies the  pointwise estimates
\bear
\label{lemma8.3a}
c_M^{-1} |\xi|\ \le\  |\wh \xi|  \ \le\   c_M |\xi|, \qquad |\nabla \wh \xi|   \ \le\   c_M  \big(|\nabla \xi|+|\xi|\big)
\eear
whenever $|\la|\le \la_M$ small. Furthermore, if  $\xi=\xi^V$ then $(\wh \xi\; )^N=0$. 
\end{lemma} 

\medskip

\begin{proof}
For each non-zero small $\mu$, equations \cite[(4.4), (4.5)]{ip} show that the cylindrical metric
on $C_\mu\cap \Omega_M^+$ is independent of $\mu$ (the ratio $g_\mu/g_0$ of the metrics is $r^2\rho^{-2}\left(1+|\frac{\mu}{z}|^2\right) = 1$). 
  On the target, the corresponding formula shows that the (smooth)   metrics $g_X$  on $X$ and $g_\la$ on  $Z_\la$ have the  form
\bear
\label{8.metricformula}
g_X\ =\ (g_V+  dx d\bar{x})+ O(R) \hspace{1cm}
g_\la\ =\ \left[ g_V + \left(1+\tfrac{|\la|^2}{|x|^4}\right)\, dx d\bar{x}\right]+ O(R),
\eear
where $g_V$ is the metric on $V$,  and  $R^2=|x|^2+|y|^2$.   Using the formula for the Christoffel symbols, one sees that the difference of the corresponding Levi-Civita connections is a 1-form $\alpha\, dx$  on $X$ with $|\alpha|\le c(1+ |\la|^2 |x|^{-5})$.

As in \eqref{new6.14}--\eqref{FdFbounds},    the coordinates of each approximate map $F$ satisfy $|x| \sim |z|^s$ and $|y| \sim |w|^s$ and $xy=\la$. Thus the image  of $\Omega_M^+$  lies in the  region $Z_\la$ where $|y|\le c_1(M) |x|$ for a constant $c_1(M)$ independent of $\la$.  In this region, the metrics \eqref{8.metricformula} are uniformly equivalent with a similar constant $c_2(M)$, giving the first part of \eqref{lemma8.3a}.  Furthermore, the covariant derivatives are related by 
$$
\nabla\wh\xi = \nabla\xi +  \alpha_F\xi \qquad \mbox{with } |\alpha_F\xi|\ \le\  c_3(M)\,(1+|\la|^ 2|x|^{-5})\cdot  \left(|df^N|+|dF^N|\right)\, |\xi|.
$$
Noting that  $xy=\la$ and $|x| \sim |z|^s\sim \rho^s$, the term $|\la|^ 2|x|^{-5}$ is dominated by $\left|\frac{y}{x}\right|^2\rho^{-s}\le c^2_1(M)\rho^{-s}$.  We also have 
 $|dF^N|\le c\rho^s$  by \eqref{FdFbounds} and the  bound  $|df^N|\le c\rho^s$ obtained similarly by taking $\beta=0$ in \eqref{new6.14}.  The last part of  \eqref{lemma8.3a} follows.
\end{proof}

\bigskip

To define cutoff functions, consider the central annular region $\Omega_M$ of $C_\mu$ defined by 
\bear
\label{8.Mregion}
M^{-1} \le\ \l|\frac{w}{z} \r|\ \le \ M.
\eear
In cylindrical coordinates  (defined by $z=\sqrt{|\mu|}e^{t+i\theta}$), this is a region of length $\log (1+M^2)$.  Thus we can choose a smooth  cutoff function $\phi_M(z,w)$ that vanishes for $|w|>M|z|$,  is equal to 1 for $M|w|\le |z|$, and satisfies $0\le\phi_M\le 1$ and 
\bear
\label{8.logphi}
|d\phi_M|\le \frac{1}{|\log M|}.
\eear
To maintain symmetry,  we  can also assume  (after appropriately symmetrizing) that
\best
\phi_M(z, w) +\phi_M(w, z)= 1.
\eest 
\smallskip

With this setup,  the maps $\gamma_F$, $\pi_F$ and $\Gamma_F$ in Diagram~\eqref{diagram8.1} are defined as follows.

\medskip

\noindent{\bf  $\bullet$ The map $\pi_F:\F_F \ra \F_f$.} \   The map \eqref{D.8.hat.map} with $q=1$ has an inverse
$$
\tau^+: \Omega^{0,1}(C^+_\mu, F^* TZ_\la) \ra \Omega^{0,1}(C_1(\mu), f^* TX).
$$
Then each $F^* TZ_\la$-valued (0,1)-form  $\eta$ on $C_\mu$  restricts to a form $\eta^+$ on $C^+_\mu$, and we define $\pi_F(\eta)$ on $C_1$ by 
\best
(\pi_F(\eta))(z)=\begin{cases} \tau^+\eta^+(z) &\mbox{ for }|z| > |\mu|^{1/2}
\\0 &\mbox{ for } |z|\le |\mu|^{1/2}.
\end{cases}
\eest
The restriction of $\pi_F(\eta)$ to $C_2$ is defined symmetrically.

\bigskip

\noindent{\bf  $\bullet$ The map $\gamma_F:\F_f \ra \F_F$.}  This map  takes a $f^* TZ_0$-valued (0,1)-form $\eta$ on $C_0$ to a $F^* TZ_\la$-valued (0,1)-form $\eta$ on $C_\mu$.  It is given by
 \bear
 \label{8.defgamma_F}
\gamma_F(\eta) \ =\ \phi_M \,\wh\eta_1 \ +\  (1-\phi_M) \,\wh\eta_2
 \eear
 where  $\wh \eta_1$ and $\wh \eta_2$ are defined in terms of the restrictions  $\eta|_{C_1} = \xi_1\, d\bar{z}$ and  $\eta|_{C_2} = \xi_2\, d\bar{w}$ by 
 $\wh \eta_1=\wh\xi_1\, d\bar{z}$  and $\wh \eta_2= \wh\xi_2\, d\bar{w}$ inside each coordinate chart $(z,w)$, and $\gamma_F=id$ outside. 

\bigskip

\noindent{\bf $\bullet$  The map $\Gamma_F:\E_f \ra \E_F$.}  This map  takes a section of $f^* TZ_0$  on $C_0$ to a section of  $F^* TZ_\la$ on $C_\mu$, and a variation $h$ in the complex structure of $C_0$ to a variation in  the complex structure of $C_\mu$.  It is given by
 \bear
 \label{8.defGamma_F}
\Gamma_F(\xi, h_0) \ =\ \big(\phi_M \,\wh\xi_1 +  (1-\phi_M) \,\wh\xi_2,\    h_\mu     \big)
 \eear
 where  $\wh \xi_1$ and $\wh \xi_2$ are obtained from  the restrictions  $\xi|_{C_1} = \xi_1$ and  $\xi|_{C_2} = \xi_2$ inside these neighborhoods, and $h_\mu=(h_0, 0)$ in the notation of \cit{4.9}.   Again,  $\Gamma_F$ extends outside as $\Gamma_F=id.$  Thus,  in the notation of  \cite[(7.3), (7.4)]{ip}, $\Gamma_F$ is a map
\best
\Gamma_F:  L_{1;s}(f^*TZ_0) \oplus T_{C_1}\wt\M\oplus T_{C_2}\wt \M\ra 
 L_{1;s}(F^*TZ_\la) \oplus T_{C_\mu}\ov\M_{g, n}.
\eest

\bigskip

\begin{cor} The maps $\pi_F$,  $\gamma_F$ and  $\Gamma_F$ satisfy  
$$
\pi_F \gamma_F =id, 
$$ and for small $\la$
\bear\label{8.d.pi}
 \! \pi_F\eta\!_0 \le 2 \! \eta\!_0 \qquad  \! \gamma_F\eta\!_0 \le c_M \! \eta\!_0 \qquad \! \Gamma_F (\xi, h)\!_1 \le c_M\! (\xi,  h)\!_1,  
 \eear
where $c_M$ depends only on the constant $M$ in \eqref{8.Mregion}.
 \end{cor}
 
 \begin{proof}
 The equation  $\pi_F \gamma_F =id$ follows directly from the definitions of $\pi_F$ and $\gamma_F$.   As in first paragraph of the proof of Lemma~\ref{lemma8.3},  the projection $C_\mu\to C_1$ is an isometry in  the region where $\rho\le 1$, and  $g_\la$ is greater than $g_X$ on its image. It  follows that   the operator norm of $\pi_F$ is at most 2 for small $\la$.

Similarly,  \eqref{8.defgamma_F}, the fact that $0\le\phi_M\le 1$,  and Lemma~\ref{lemma8.3}  show that 
$$
\!\gamma_F(\eta)\!_0 \ \le \ \!\wh\eta_1\!_0 \ +\!\wh\eta_2\!_0\ \le\ c_M\!\eta\!.
$$
Using  \eqref{8.defGamma_F} in exactly the same way, we also have 
$$
\!\Gamma_F(\xi, h)\!_0  \ \le \ \!\wh\xi_1\!_0 \ +\!\wh\xi_2\!_0 +\| h\| \ \le\ c_M \! (\xi, h)\!_0.
$$
Differentiating \eqref{8.defGamma_F} and again applying 
Lemma~\ref{lemma8.3},  yields the last inequality in \eqref{8.d.pi}.
 \end{proof}

 \bigskip

 The next lemma shows that the  difference $\D_F \Gamma_F-\gamma_F \D_f$ can be made small.  The statement again  involves the constant $M$  in the bounds  \eqref{8.Mregion}  and  \eqref{8.logphi} associated with the cutoff functions $\phi_M$.

\begin{lemma}\label{L.approx.com} 
 Fix the compact  subset ${\cal K}_{\de_0}$ of $\M_s^V(X) \ti_\ev \M_s^V(Y)$  of  $\delta_0$-flat maps.   For any $\ep>0$, there exits  a slope $M=M_\ep>1$ and a $\la_M>0$  such that each approximate map $F$ constructed from  $f\in {\cal K}_{\de_0}$ with $|\la |\le \la_M$ satisfies   
\best
\! (\D_F \Gamma_{F, M}- \gamma_F \D_f)(\xi, h)\!_0\ \le\   \ep \! (\xi, h)\!_1
\eest
for all $(\xi, h)\in \E_f$. 
 \end{lemma} 

\begin{proof} We use the set-up of  Lemma 9.2 above, except that we do not make the assumption that  $\D_f(\xi_0, h_0)=0$ in \eqref{9.2second}.  Outside the region  $B=\bigcup B_k$,  $\Gamma_{F, M}$ and $\gamma_F$ are both the identity for 
$|\mu|\le M^{-2}$.  Thus the discussion from \eqref{9.2second} to \eqref{1.15B} implies  the bound 
\best
\!(\D_F \Gamma_F- \gamma_F \D_f)(\xi, h)\!_{0,C_0 \setminus B}  \ \le\  c |\la|^{\frac1{4}} \ \! (\xi, h)\|_1.
\eest
Next restrict attention to the  region $B_k$ around one node of $C_0$, where $\rho\le 2 |\mu|^{1/4}$,   
and consider a deformation $(\xi, h)$ on  $C_1$ (an identical analysis applies on $C_2$).  Then  $\xi\in \Gamma(C_1, f^*TX)$  lifts by \eqref{D.8.hat.map}  to  $\wh \xi\in \Gamma(C^+_\mu, F^* TZ_\la)$:  this is the identification implicitly used in \cit{9.5} and in \eqref{9.2second}.  With this notation, \eqref{9.2second} can be written as
\best
| \D_F(\wh \xi, h)-\wh{\D_f(\xi, h)}|\ \le\  |L_F\wh \xi - \wh {L_f \xi}|+c |J_F- J_f||\xi| + |dF- df| |h|.
\eest
For the following estimates, we restrict attention to  the  annular subregion  $A^+_M\subset B_k$ where 
 \bear
\label{8.5.A}
|w/z|\le M \qquad \mbox{ and  thus } \quad |z|\le \rho \le |z|\sqrt {1+M^2}.
\eear
In this subregion,  the   $C^1$ norm of $F-f$  is   $O(\rho)$,    as shown in the proof of  Lemma~6.9.  Using \eqref{9.2second} and \eqref{9.2.third}  we obtain 
$$
| \D_F(\wh \xi, h)-\wh{\D_f(\xi, h)}|\ \le\  c\rho \big(|\nabla \wh \xi|+|\wh \xi| +|\nabla  \xi|+| \xi| +|h|\big).
$$
Lemma \ref{lemma8.3} then shows that we may remove the hats on the righthand side, giving
 \bear\label{8.core}
 | \D_F(\wh \xi, h)-\wh{\D_f(\xi, h)}|\ \le\  c\rho \big(|\nabla  \xi|+| \xi| +|h|\big). 
\eear
Here the left-hand side is regarded as a function of $z$ and $w=\mu/z$ on $C_\mu\cap A^+_M$, while $\xi$ and $h$ are functions of $z$ on the corresponding region of $C_1$,  and $\rho^2=|z|^2+|w^2|$.   

 As in previous lemmas, we need separate bounds for the normal components.  First,   \eqref{8.core} for $\xi=\xi^N$ and $h=0$ gives
\bear
\label{8.222}
\left| \D_F(\wh \xi^N, 0)^N-\wh{\D_f(\xi^N,0)}\right|\ \le\  c\rho \big(|\nabla  \xi^N|+| \xi^N|\big).
\eear
On the other hand, if $\xi=\xi^V$ is tangent to $V$, Lemma \ref{lemma8.3} shows that its lift $\wh \xi$ is also tangent to $V$.  Writing $D_F(\xi,h)=L_F\xi+J_FdFh$, we can apply \eqref{newLemma7.2eq} and the argument made before \eqref{est.DN} to obtain
$$
|\D_F (\wh {\xi^V}, h)^N| \ \le  \ c \rho^s\l(|\nabla \wh{\xi^V}|+ |\wh {\xi^V}|+ |h|\r).
$$
By the same argument, a similar inequality holds with $D_F$ replaced by $D_f$ and $\wh {\xi^V}$ by $\xi^V$; together these give
$$
|\D_F (\wh {\xi^V}, h)^N-(\wh{\D_f (\xi^V, h)})^N| \ \le  \ c \rho^s\l(|\nabla \xi^V|+ |\xi^V|+ |h|\r)
$$
after again using \eqref{lemma8.3a} to remove hats on the right. Combining this with \eqref{8.222} gives 
\bear\label{8.core2}
|\D_F (\wh {\xi}, h)^N-(\wh{\D_f (\xi, h)})^N| \ \le  \  c \rho (|\nabla \xi^N|+|\xi^N|)+ c \rho^s\l(|\nabla \xi^V|+ |\xi^V|+ |h|\r).
\eear

Now set   $\Psi(\xi,h) = D_F (\wh {\xi}, h)-\wh{\D_f (\xi, h)}$.  With this notation,  we can combine  \eqref{8.core} and  \eqref{8.core2}, and then decompose $\xi$ into $(\zeta, \ov\xi)$ as in \cit{6.9}, noting that $|\nabla \xi|+|\xi|\le |\nabla \zeta|+|\zeta|+|\ov\xi|$ as before \eqref{9.new9.6}.   The result is 
\begin{eqnarray*}
|\Psi(\xi,h)| +\rho^{1-s}|(\Psi(\xi,h))^N| &  \le & c \rho \Big(|\nabla \xi|+|\xi| +|h| + \rho^{1-s} (|\nabla \xi^N|+ |\xi^N|)\Big) \\
& \le & c \rho \Big(|\nabla \zeta|+|\zeta| + |\bar{\xi}| +|h| + \rho^{1-s} (|\nabla \xi^N|+ |\xi^N|)\Big).
\end{eqnarray*}
Multiplying by $\rho^{-\de/2}$ and computing the integral \eqref{new6.9} over $A^+_M \cap C_\mu$, where  $\rho\sim |z|$ by \eqref{8.5.A} and $|\la|\sim |\mu|^s$, one sees that 
\begin{eqnarray}\label{8.la/5sbound}
\!\D_F (\wh \xi_1, h_1)- \wh{ \D_f(\xi_1, h_1)}\!_{0,A^+_M\cap C_\mu}   &  \le  & c_M |\mu|^{1/4(1-\de/2)} \| \xi_1, h_1\|_1 \nonumber \\
&  \le  & c_M |\la|^{\frac{1}{5|s|}} \| \xi_1, h_1\|_1
\end{eqnarray}
for all pairs $(\xi_1, h_1)$ on   $C_1$.   A similar estimate holds for pairs $(\xi_2, h_2)$ on  $C_2$. 

To complete the proof, recall that both $\Gamma_F$ and $\gamma_F$ are obtained by splicing  in the region $\Omega_M$ of \eqref{8.Mregion}.  Using  \eqref{8.defGamma_F}, the formula  $\D_F(\xi, h)=L_F\xi+JdFh$ and the Leibnitz rule, we obtain 
\best
\left| \D_F\Gamma_F(\xi,h)- \phi_M \D_F(\wh \xi_1, h_\mu) -(1-\phi_M) \D_F (\wh \xi_2, h_\mu)\right|  \ \le \ 
 |d\phi_M| \cdot |\xi| \ \le\  |\log M|^{-1} |\xi|.
\eest
Combining this  with \eqref{8.defgamma_F} and \eqref{8.la/5sbound} gives the following uniform estimate: 
$$
  \! (\D_F \Gamma_{F, M}- \gamma_F \D_f)(\xi, h)\!_0\ \le  \ 
  \left(c_M\la^{\frac 1 {5|s|}}+ c|\log M|^{-1}\right) \! (\xi, h)\!_1.
$$
The lemma follows by first choosing $M$ so that $c|\log M|^{-1}\le \ep/2$, then choosing  $\la_M$ so that $c_M\la_M^{\frac{1}{5|s|}}\le \ep/2$.
\end{proof}

\bigskip

We are now able to define the approximate inverse of $\D_F$. Recall that the linearization $\D_f$ of a regular map is onto.
Fix the compact set ${\cal K}={\cal K}_{\delta_0}$   of $\de_0$-flat regular maps.  Then we can choose a family $P_f$ of partial right inverses of $D_f$ which are uniformly bounded 
\best
\| P_f \eta\|_0 \ \le\  K \| \eta\|_1.
\eest 
on  ${\cal K}$. Recall that the operator norm of $\pi_F$ is at most 2.

\begin{lemma} 
\label{AFexistsLemma} In the above context,  there exist positive  constants $C, M$ and $\la_0$ such that for any approximate map $F$ (obtained from $f\in \cal K$ for 
$|\la| \le \la_0$), the operators \best
A_F= \Gamma_{F, M}\circ  P_f\circ  \pi_F: \E_F\to \F_F
\eest
obtained from  splicing in region \eqref{8.Mregion} satisfy
\bear
\label{8.defaprinverse}
\! D_F A_F \eta - \eta\!_0\  \le\  \tfrac12 \! \eta\!_0 \qquad\mbox{and} \qquad  \! A_F\eta\!_1 \ \le\   C \! \eta\!_0
\eear
for all $\eta\in\F_F$. 
\end{lemma}
\begin{proof}  Write $D_F A_F-I$ as
 \best
 D_F\circ \Gamma_F\circ  P_f\circ  \pi_F - id &=&  (D_F \Gamma_F- \gamma_F D_f) \circ  P_f\circ  \pi_F +  \gamma_F D_f \circ  P_f\circ  \pi_F -id
 \\  
&=&
  (D_F \Gamma_F- \gamma_F D_f) \circ  P_f\circ  \pi_F.
\eest
We know that $\! \pi_F\!\le 2$ and $\! P_f\!\le K$ are uniformly bounded on the compact $\cal K$.  Hence there is a bound on the operator norm:
\best
\| D_F A_F-I\| \ \le\  \| D_F \Gamma_F- \gamma_F D_f\| \cdot \| P_f \| \cdot \|\pi_F\| \ \le\  2K  \| D_F \Gamma_F- \gamma_F D_f\|.
\eest
Now take $\ep= \frac 1{4K}>0$ in Lemma~\ref{L.approx.com} to obtain the first inequality of \eqref{8.defroughinverse}.  This choice of $\ep$   fixes 
the slope $M=M_\ep$ in Lemma~\ref{L.approx.com}.  With this choice,   $\Gamma_{F, M}$ are bounded, with a  bound that  depends on $M$, and hence
\best
\| A_F\| \ =\   \| \Gamma_{F, M}\circ  P_f\circ  \pi_F\| \ \le\   \| \Gamma_{F, M}\| \cdot \| P_f\| \cdot\| \pi_F\| \ \le\   2K  \| \Gamma_{F, M}\|.
\eest
\end{proof}

\vspace{5mm}

\section{{\bf \large Typographical errors}} 

The following typographical errors in
\cite{ip} have no consequences, but may cause confusion.
{\small
\begin{itemize}
\item Line after (5.1): insert ``after passing to a subsequence''.
\item In (5.18), delete $+1/3$.
\item In (1.11), the second $+$ should be a $-$.
\item Page 946, line 4:  $J(\nabla_{\nu(w)}J)\xi \to (\nabla_{J\nu(w)}J)\xi$. One can also note that the tensor $\widehat{\nabla}J$ is zero by  (C.7.5) of \cite{ms2}  for compatible structures $(\w, J, g)$.
\item  Page 1003, line 8:  (10.6) $\mapsto$ (10.11).
\end{itemize}

\vspace{1cm}


 \end{document}